\documentclass[english, 10pt]{amsart}

\usepackage{babel}
\usepackage{amstext}
\usepackage{amsmath}
\usepackage{amssymb}
\usepackage{amsfonts}
\usepackage{latexsym} 
\usepackage{ifthen}
\usepackage{xypic}
\xyoption{all}
\usepackage{verbatim}

\newcommand{\id}{{\rm id}}

\newtheorem{lemma1}{}[section]

\newenvironment{lemma}{\begin{lemma1}{\bf Lemma.}}{\end{lemma1}}

\newenvironment{theorem}{\begin{lemma1}{\bf Theorem.}}{\end{lemma1}}
\newenvironment{proposition}{\begin{lemma1}{\bf Proposition.}}{\end{lemma1}}
\newenvironment{corollary}{\begin{lemma1}{\bf Corollary.}}{\end{lemma1}}
\newenvironment{remark}{\begin{lemma1}{\bf Remark.}\rm}{\end{lemma1}}
\newenvironment{definition}{\begin{lemma1}{\bf Definition.}}{\end{lemma1}}

\newenvironment{conjecture}{\begin {lemma1}{\bf Conjecture.}}{\end{lemma1}}

\newenvironment{remark*}{{\bf Remark.}}{}
\newenvironment{example*}{{\bf Example.}}{}

\newcommand{\Q}{\ensuremath{\mathbb{Q}}}

\newcommand{\C}{\ensuremath{\mathbb{C}}}
\newcommand{\N}{\ensuremath{\mathbb{N}}}

\newcommand{\holom}[3]{\ensuremath{#1\colon #2  \rightarrow #3}}
\newcommand{\fibre}[2]{\ensuremath{#1^{-1} (#2)}}

\makeatletter
\ifnum\@ptsize=0  \fi \ifnum\@ptsize=2  \fi 

\DeclareMathOperator*{\reg}{reg}

\DeclareMathOperator*{\Nlc}{Nlc}
\DeclareMathOperator*{\mult}{mult}
\DeclareMathOperator*{\supp}{supp}
\DeclareMathOperator*{\gen}{gen}
\DeclareMathOperator*{\Vol}{Vol}
\DeclareMathOperator*{\Exc}{Exc}

\title{Singularities of varieties admitting an endomorphism} 
\date{\today}

\author{Ama\"el Broustet}
\author{Andreas H\"oring}

\subjclass[2000]{14B05, 14E30, 14J40}
\keywords{endomorphism, non-lc locus, log-canonical model}

\address{Ama\"el Broustet, Universit{\'e} Lille 1, UMR CNRS 8524,
UFR de math{\'e}matiques, 59 655 Villeneuve d'Ascq CEDEX.}
\email{amael.broustet@math.univ-lille1.fr}

\address{Andreas H\"oring, Universit{\'e} Pierre et Marie Curie, Institut de math{\'e}matiques de Jussieu,
Projet Topologie et g{\'e}om{\'e}trie alg{\'e}briques, Case 247,  4 place Jussieu, 75005 Paris, France}
\email{hoering@math.jussieu.fr}

\begin{document}

\begin{abstract}
Let $X$ be a normal variety such that $K_X$ is $\Q$-Cartier, and let \holom{f}{X}{X} be a finite surjective morphism 
of degree at least two.
We establish a close relation between the irreducible components of the locus of singularities that are not
log-canonical and the dynamics of the endomorphism $f$.
As a consequence we prove that if $X$ is projective and $f$ polarised, then $X$ has at most log-canonical
singularities.
\end{abstract}

\maketitle

\section{Introduction}
\subsection{Main result}

Let $X$ be a normal variety and let \holom{f}{X}{X} be an endomorphism, 
i.e. a finite surjective morphism
of degree $\deg(f)>1$. If $X$ is projective, an abundant literature 
\cite{Bea01, Fuj02, Ame03,  FN07, Nak08, AKP08, NZ10, Zha10} 
shows that the existence of an endomorphism imposes strong restrictions on the global geometry of $X$. 
In this paper we address the question
if the existence of an endomorphism also imposes restrictions on the local geometry, i.e. restrictions
on the nature of the singularities. In a recent paper Boucksom, de Fernex and Favre introduce
the volume $\Vol(X, x)$ of an isolated singularity. Using this invariant they give 
a precise answer to our question for isolated singularities. 

\begin{theorem} \label{theorembdf} \cite[Thm.B]{BDF12}
Let $X$ be a normal variety with isolated singularities, and let $\holom{f}{(X, x)}{(X, x)}$ be an endomorphism of degree $\deg(f)>1$. Then we have $\Vol(X, x) = 0$.

If $K_X$ is $\Q$-Cartier then X has log-canonical singularities, and it furthermore has klt singularities if $f$ is not \'etale in codimension one.
\end{theorem}

Fulger \cite{Ful11} introduces a different invariant $\Vol_F(X, x)$ associated to an isolated singularity and proves the analogue of Theorem \ref{theorembdf} for $\Vol_F(X, x)$. Let us note that
$$
\Vol(X, x) \geq \mbox{Vol}_F(X, x)
$$
and equality holds if $K_X$ is $\Q$-Cartier.

In this paper we will consider varieties such that $K_X$ is $\Q$-Cartier, but the singularities are not isolated.
In this case $X$ is not necessarily log-canonical: if $Y$ is any normal variety such that $K_Y$ is $\Q$-Cartier and $E$ an elliptic curve, 
then
$X:= Y \times E$ admits the endomorphism $f := \id_Y \times g$ with $g$ the multiplication by $m \in \N$.
However we can establish a close relation between 
the irreducible components of the non-lc locus and the dynamics of the endomorphism:

\begin{theorem} \label{theoremmainlocal}
Let $X$ be a normal variety such that $K_X$ is $\Q$-Cartier, and let $\holom{f}{X}{X}$ be an endomorphism of degree $\deg(f)>1$.

Let $Z$ be an irreducible component of $\Nlc(X)$. Then (up to replacing $f$ by some iterate) $Z$ is totally invariant.
In this case $Z$ is not contained in the ramification divisor $R$,
and the induced endomorphism \holom{f|_Z}{Z}{Z} satisfies 
$$
\deg(f|_Z)=\deg(f).
$$
\end{theorem}

Since we suppose $\deg(f)>1$ the last part of this statement shows that $Z$ cannot be a point, so we recover the $\Q$-Cartier case
of Theorem \ref{theorembdf}. 
If $X$ is projective we can consider the particularly interesting class of polarised endomorphisms, i.e. those endomorphisms such that there exists an ample divisor $H$ satisfying $f^* H \simeq m H$. In this case the statement 
becomes much stronger:

\begin{corollary} \label{corollarypolarised}
Let $X$ be a normal projective variety such that $K_X$ is $\Q$-Cartier, 
and let $\holom{f}{X}{X}$ be a polarised endomorphism of degree $\deg(f)>1$.

Then $X$ has at most log-canonical singularities. Moreover $X$ is klt near the ramification divisor $R$.
\end{corollary}

\subsection{Technique and generalisations}
The proof of our main result comes in two steps. In the first step we use a classical computation
describing the behaviour of log-discrepancies under finite morphisms \cite[Prop.5.20]{KM98} 
to prove that all the irreducible components of the non-lc locus are totally invariant. 
In the second step we use an idea introduced by Nakayama in his inspiring
preprint \cite{Nak08} on endomorphisms of normal surfaces: if $\mu: Y \rightarrow X$ is the log-canonical model
(cf. Definition \ref{definitionlcmodel}), the endomorphism $f$ lifts to a (rational) endomorphism 
$g$ of $Y$. We can then study the geometry of the ramification divisors along certain $\mu$-exceptional divisors
to deduce our result. 

Our proof actually works more generally for log pairs $(X, \Delta)$ such that $K_X+\Delta$ is $\Q$-Cartier and 
a logarithmic ramification formula holds. In this paper we focus on the geometrically most interesting
case where the boundary $\Delta$ is a totally invariant Weil divisor. 

\begin{theorem} \label{theoremmainlocalpair}
Let $X$ be a normal variety, and let $\holom{f}{X}{X}$ be an endomorphism of degree $\deg(f)>1$.
Let $\Delta$ be a reduced effective totally invariant Weil divisor such that $K_X+\Delta$ is $\Q$-Cartier.

Let $Z$ be an irreducible component of $\Nlc(X, \Delta)$. Then (up to replacing $f$ by some iterate) $Z$ is totally invariant.
In this case we have $Z \not\subset R_\Delta$ where $R_\Delta$ is the logarithmic ramification divisor,
and the induced endomorphism \holom{f|_Z}{Z}{Z} satisfies 
$$
\deg(f|_Z)=\deg(f).
$$
\end{theorem}

Theorem \ref{theoremmainlocalpair} is simply the case $\Delta= \emptyset$ in the preceding statement.

Let us note that the existence of log-canonical models has been proven recently
by Odaka and Xu \cite{OX12} for pairs $(X, \Delta)$ such that $K_X+\Delta$ is $\Q$-Cartier.
If log-canonical models exist in general\footnote{The existence of log-canonical models would be a consequence of the MMP, including the
abundance conjecture.}, it seems plausible that our results can be generalised to arbitrary normal varieties.

\begin{conjecture} \label{conjecturegeneral}
Let $X$ be a normal variety, and let $\holom{f}{X}{X}$ be an endomorphism of degree $\deg(f)>1$.
Suppose that $X$ admits a log-canonical model $\holom{\mu}{Y}{X}$.
Let $Z$ be an irreducible component of $\mu(E^{lc}_\mu)$, where 
$E^{lc}_\mu$ is the sum of all the $\mu$-exceptional prime divisors taken with coefficient one.

Then (up to replacing $f$ by some iterate) $Z$ is totally invariant. In this case $Z$ is not contained in the ramification divisor $R$,
and the induced endomorphism \holom{f|_Z}{Z}{Z} satisfies 
$$
\deg(f|_Z)=\deg(f).
$$
If moreover $X$ is projective and $f$ is polarised, then $\mu$ is an isomorphism in codimension one.
\end{conjecture}

This statement would also generalise Theorem \ref{theorembdf} since we can prove that an isolated singularity has volume zero if and only if the log-canonical model (if it exists) is an isomorphism in codimension one, cf. Proposition \ref{propositionvolume}.

{\bf Acknowledgements.} We thank S. Druel and Y. Gongyo for pointing out some crucial references.
The authors are partially supported by the ANR project CLASS\footnote{\tiny ANR-10-JCJC-0111}. A.B is partially supported by the ANR project MACK\footnote{\tiny ANR-10-BLAN-0104} and the Labex CEMPI\footnote{\tiny ANR-11-LABX-0007-01}.

\section{Notation and basic results}

We work over the complex field $\C$, topological notions always refer to the Zariski topology.
For general definitions we refer to Hartshorne's book \cite{Har77}.
We will use standard terminology and results 
of the minimal model program (MMP) as explained in \cite{KM98} or \cite{HK10}. 
A variety is an integral scheme of finite type over $\C$. For $D$ a $\Q$-Weil divisor on a normal variety $X$,
we denote by $\supp(D)$ its support.

\subsection{Singularities of pairs}

Let $X$ be a normal variety, and let \holom{\mu}{X'}{X} be a proper birational morphism from 
a normal variety $X'$. If $\Delta \subset X$ is a $\Q$-Weil divisor, we denote by 
$\mu_*^{-1}(\Delta)$ its strict transform. 

A log-pair is a tuple $(X, \Delta)$ where $X$ is a normal variety
and $\Delta=\sum_i d_i \Delta_i$ is a $\Q$-Weil divisor on $X$ with $d_i \leq 1$ for all $i$. 
We say that 
the pair $(X, \Delta)$ is lc (resp. klt)\footnote{Note that we do not assume that the boundary
divisor $\Delta$ is effective, so some authors would say that such a pair is sub-lc (resp. sub-klt).
We follow the notation of \cite{KM98}.} if 
$K_X+\Delta$ is $\Q$-Cartier and
for every proper birational morphism 
\holom{\mu}{X'}{X} from a normal variety $X'$
we can write  
$$
K_{X'}+\mu_*^{-1}(\Delta) = \mu^* (K_X+\Delta) + \sum_j a(E_j, X, \Delta) E_j,
$$
where the divisor $E_j$ are $\mu$-exceptional and $a(E_j, X, \Delta) \geq -1$ (resp. $a(E_j, X, \Delta)>-1$) for all $j$.
If the pair $(X, \Delta)$ is log-canonical, we say that a subvariety $Z \subset X$ is an lc centre if there exists
a morphism \holom{\mu}{X'}{X} as above and a $\mu$-exceptional divisor $E$ such that $E \twoheadrightarrow Z$ 
and $a(E, X, \Delta)=-1$.

\begin{definition} \label{definitionnlc}
Let $(X, \Delta)$ be a log-pair such that $K_X+\Delta$ is $\Q$-Cartier. 
The non-lc locus $\Nlc(X, \Delta)$ is the smallest closed set $W \subset X$
such that $(X \setminus W, \Delta|_{X \setminus W})$ is lc. 
\end{definition}

\begin{definition} \label{definitionlcmodel}
Let $(X, \Delta)$ be a log-pair such that $\Delta \geq 0$.
A log-canonical model of the pair $(X, \Delta)$
is a proper birational morphism
$$
\holom{\mu}{Y}{X}
$$
such that if we set
$$
\Delta_Y := \mu_*^{-1}(\Delta)+E^{lc}_\mu,
$$
where $E^{lc}_\mu$ is the sum of all the $\mu$-exceptional prime divisors taken with coefficient one, the pair $(Y, \Delta_Y)$ is log-canonical and $K_Y+\Delta_Y$ is $\mu$-ample.
\end{definition}

\begin{remark} \label{remarklcmodel}
\begin{enumerate}
\item[a)] If a pair $(X, \Delta)$ admits a log-canonical model, it is unique up to isomorphism \cite[Prop.2.3]{OX12}.
\item[b)] Suppose now that $\Delta \geq 0$ and $K_X+\Delta$ is $\Q$-Cartier. Then $(X, \Delta)$ admits a log-canonical model \cite[Thm.1.1]{OX12}.
Moreover the $\mu$-exceptional locus has pure codimension one \cite[Lemma 2.4]{OX12}.
If we write 
\begin{equation} \label{lcmodel}
K_Y+\Delta_Y = \mu^* (K_X+\Delta) + \Delta_Y^{>1},
\end{equation}
then $\Delta_Y^{>1}$ is antieffective and $\supp \Delta_Y^{>1}=\Exc(\mu)$ (ibid). 
By the definition of $\Delta_Y$ we have $\supp \Delta_Y^{>1} \subset \Delta_Y$.
Note also that since $K_X+\Delta$ and $K_Y+\Delta_Y$ are $\Q$-Cartier, the divisor $\Delta_Y^{>1}$ is $\Q$-Cartier.
\end{enumerate}
\end{remark}

The following proposition establishes the link between Conjecture \ref{conjecturegeneral} and Theorem~\ref{theorembdf}.

\begin{proposition} \label{propositionvolume}
Let $X$ be a normal variety with singular locus a point $x$. 
Assume that $X$ has a log-canonical model $\mu : (Y, \Delta_Y) \rightarrow X$.

Then $\Vol(X, x) = 0$ if and only if $\mu$ is an isomorphism in codimension $1$.
\end{proposition}

For the proof of this statement we will use the tools and terminology of \cite{BDF12}: given a canonical divisor $K_X$ on $X$, there is a unique canonical divisor $K_{X_\pi}$, for each birational model $\pi \colon X_\pi \to X$, 
with the property that $\pi_*K_{X_\pi} = K_X$. 
Thus we obtain a canonical $b$-divisor $K_{\mathfrak X}$ over $X$. Boucksom, de Fernex and Favre
define the nef envelope $\mbox{Env}_X(-K_X)$ of the Weil divisor $-K_X$ as the largest
nef Weil $b$-divisor $Z$ that is both relatively nef over $X$ and satisfies $Z_X \leq -K_X$.
The log-discrepancy $b$-divisor $A_{{\mathfrak X}/X}$ is then defined by  
\begin{equation} \label{bdivisors}
A_{{\mathfrak X}/X} = K_{\mathfrak X} + 1_{{\mathfrak X}/X} + \mbox{Env}_X(-K_X),
\end{equation}
where the trace of $1_{\mathfrak X/X}$ in any model is equal to the reduced exceptional divisor over $X$.

\begin{proof}[Proof of Proposition \ref{propositionvolume}] 
{\em Suppose that $\Vol(X, x) = 0$.} We will argue by contradiction and suppose that the divisor $\Delta_Y$ is not zero.
Let \holom{\nu}{Z}{Y} be a dlt-model of the log-canonical pair $(Y, \Delta_Y)$ \cite[Thm.10.4]{Fuj11}, 
i.e. $\nu$ is a birational morphism
from a normal $\Q$-factorial variety $Z$ such that if we denote by $B$ the $\nu$-exceptional divisors taken with
coefficient one and set
$$
\Delta_Z := \nu_*^{-1}(\Delta_Y)+B,
$$
then the pair $(Z, \Delta_Z)$ is dlt and we have
$$
K_Z+\Delta_Z = \nu^* (K_Y+\Delta_Y).
$$
Set $\varphi:=\mu \circ \nu$. Then the divisor $K_Z+\Delta_Z$ is $\varphi$-nef and its restriction to any irreducible 
component of $\nu_*^{-1}(\Delta_Y)$ is nef and big.

The trace of the equation \eqref{bdivisors} on $Z$ is 
$$
\nu_*^{-1}((A_{{\mathfrak X}/X})_Y) = K_Z + \Delta_Z + (\mbox{Env}_X(-K_X))_Z.
$$
Indeed $\Delta_Z$ is the union of all the $\varphi$-exceptional divisors taken with multiplicity one, so
$(1_{{\mathfrak X}/X})_Z = \Delta_Z$. Moreover all the $\nu$-exceptional divisors have log-discrepancy $0$,
so $(A_{{\mathfrak X}/X})_Z$ is just equal to the strict transform of $(A_{{\mathfrak X}/X})_Y$.

By \cite[Lemma 2.10]{BDF12} the restriction of $(\mbox{Env}_X(-K_X))_Z$ to any $\varphi$-exceptional divisor is pseudoeffective,
so the restriction of $\nu_*^{-1}(A_{{\mathfrak X}/X})_Y$ to any irreducible 
component of $\nu_*^{-1}(\Delta_Y)$ is big. 
Since $\Delta_Y$ is not zero, this implies that $\nu_*^{-1}(A_{{\mathfrak X}/X})_Y$ is not the zero divisor.
Since we have
$$
\supp \nu_*^{-1}(A_{{\mathfrak X}/X})_Y \subset \supp \nu_*^{-1}(\Delta_Y),
$$
we see that the restriction of $\nu_*^{-1}(A_{{\mathfrak X}/X})_Y$ to any irreducible component of
its support is big. By the negativity lemma (in its big version \cite[Prop.4.1]{Gra12})
this implies that $\nu_*^{-1}(A_{{\mathfrak X}/X})_Y$ is not effective.
Thus the log-discrepancy $b$-divisor $A_{{\mathfrak X}/X}$ is not effective, a contradiction to \cite[Prop.4.19]{BDF12}.

{\em Suppose that $\mu$ is an isomorphism in codimension $1$.} The variety $Y$ has log-canonical singularities,
so all the log-discrepancies are non-negative. Since $\mu$ is an isomorphism in codimension one
we see that $A_{\mathfrak X/X}$ is effective, hence $\Vol(X, x) = 0$ by \cite[Prop.4.19]{BDF12}.
\end{proof}

\subsection{Logarithmic ramification formula}

Let $\holom{f}{X_1}{X_2}$ be a finite surjective morphism between normal varieties. 
For every Weil divisor $D \subset X_2$ we define the pull-back $f^*D$ 
as the unique Weil divisor obtained by completing $f^* D|_{X_{2, \reg}}$.
If $D$ is $\Q$-Cartier of Cartier index $m$, then $f^* D$ is $\Q$-Cartier of index $m$.
The ramification divisor is defined by
\begin{equation} \label{ramification}
R := \sum_{D \subset X_2} f^* D - \supp(f^* D),
\end{equation}
where the sum runs over all prime divisors in $X_2$. 
By generic smoothness the sum is finite, so $R$ is an effective Weil divisor.
Its image $B:=f(R)$ is the branch divisor of $f$.
By the ramification formula we have 
$$
K_{X_1} = f^* K_{X_2} + R.
$$

\begin{lemma} \label{lemmalogarithmic}
Let $\holom{f}{X_1}{X_2}$ be a finite surjective morphism between normal varieties.
Let $\Delta_2$ be a reduced effective Weil divisor, and set $\Delta_1 := \supp f^* \Delta_2$. Then we have the logarithmic ramification formula
\begin{equation} \label{logarithmic}
K_{X_1} + \Delta_1 = f^* (K_{X_2}+\Delta_2) + R_\Delta,
\end{equation}
where $R_\Delta$ is an effective divisor. Moreover $\Delta_1$ and $R_\Delta$ do not have any common component. 
We call $R_\Delta$ the logarithmic ramification divisor.
\end{lemma}

\begin{proof}
Adding $\Delta_1 = f^* \Delta_2 - (f^* \Delta_2-\Delta_1)$ to the ramification formula we obtain 
$$
K_{X_1} + \Delta_1 = f^* (K_{X_2}+\Delta_2) + R - (f^* \Delta_2-\Delta_1).
$$
We claim that $R_\Delta:=R - (f^* \Delta_2-\Delta_1)$ is an effective divisor such that 
$\Delta_1$ and $R_\Delta$ do not have 
any irreducible components in common. Indeed if $W \subset \Delta_2$ is an irreducible component, we have
$$
f^* W = \sum m_i W_i,
$$
with $W_i$ the irreducible components of $f^* W$ and $m_i$ the ramification index along $W_i$. In particular
if $W$ is not in the branch divisor $B$, then $f^* W= \sum W_i$ so 
$$
\mult_{W_i} f^* \Delta_2 = \mult_{W_i} \Delta_1 = 1,
$$
and obviously no $W_i$ is contained in $R$. If $W \subset B$, then by \eqref{ramification} we have
$$
\mult_{W_i} R = \mult_{W_i} f^* W - 1 = \mult_{W_i} f^* \Delta_2 - \mult_{W_i} \Delta_1.
$$
Thus we have $\mult_{W_i} R_\Delta = 0$.
\end{proof}

\begin{remark} \label{remarklogarithmic}
If $K_{X_1}+\Delta_1$ and $K_{X_2}+\Delta_2$ are $\Q$-Cartier, then $R_\Delta$ is $\Q$-Cartier.
\end{remark}

We will also use a weak generalisation of the logarithmic ramification
formula \eqref{logarithmic} to morphisms which are only generically finite.

\begin{lemma} \label{lemmagenericallyfinite}
Let $\holom{g}{V}{Y}$ be a generically finite, projective, surjective morphism 
between normal varieties. 
Let $\Delta_Y$ be a reduced effective Weil divisor on $Y$ such that $K_Y+\Delta_Y$ is $\Q$-Cartier.
Let \holom{\eta}{V}{V_{St}} and $\holom{h}{V_{St}}{Y}$ be the Stein factorisation of $g$.
Set 
$$
\Delta_V := \eta_*^{-1}(\supp h^* \Delta_Y).
$$
Then we have
$$
K_V+\Delta_V = g^* (K_{Y} + \Delta_Y) + R_g
$$
where $R_g$ is a $\Q$-Weil divisor. Moreover 
$\Delta_V$ and $R_g$ do not have any common component. 
\end{lemma}

\begin{proof}
The morphism $h$ is finite, so by
\eqref{logarithmic} we have
$$
K_{V_{St}} + \supp h^* \Delta_Y = h^* (K_Y+\Delta_Y) + R_{\Delta_{St}},
$$
where $R_{\Delta_{St}}$ is an effective Weil divisor 
that has no common component with $\supp h^* \Delta_Y$.
The divisor $K_{V_{St}}+\supp h^* \Delta_Y-R_{\Delta_{St}}$ is $\Q$-Cartier, so we can write
$$
K_V+ \eta_*^{-1}(\supp h^* \Delta_Y-R_{\Delta_{St}}) 
= \eta^* (K_{V_{St}}+\supp h^* \Delta_Y-R_{\Delta_{St}}) + E = f^* (K_{Y} + \Delta_Y) + E
$$
where $E$ is an $\eta$-exceptional divisor.
Set now
$$
R_g := E + \eta_*^{-1}(R_{\Delta_{St}}).
$$
Since every irreducible component of $E$ is $\eta$-exceptional and $R_{\Delta_{St}}$ 
has no common component with $\supp h^* \Delta_Y$, it is clear that
$\eta_*^{-1}(\supp h^* \Delta_Y)$ has no common component with $R_g$.
\end{proof}

\subsection{Endomorphisms and Nlc-locus}

\begin{definition} \label{definitiontotallyinvariant}
Let $X$ be a normal variety, and let \holom{f}{X}{X} be an endomorphism of degree $\deg(f)>1$.
We say that a closed subset $Z \subset X$ is totally invariant if we have a set-theoretical equality $\fibre{f}{Z}=Z$.
\end{definition}

\begin{remark} \label{remarkdominates}
Let \holom{f}{X_1}{X_2} be a finite surjective morphism between normal varieties. By \cite[Cor.14.4.]{EGAIV}
the morphism $f$ is universally open. In particular if $Z \subset X_2$ is any subvariety, the induced
morphism $X_1 \times_{X_2} Z \rightarrow Z$ is open. Hence every irreducible component
of $X_1 \times_{X_2} Z$ dominates $Z$. 
\end{remark}

\begin{lemma} \label{lemmainvariance}
Let $X$ be a normal  variety, and let $\holom{f}{X}{X}$ be an endomorphism of degree $\deg(f)>1$.
Let $\Delta$ be a reduced effective totally invariant Weil divisor such that $K_X+\Delta$ is $\Q$-Cartier.
Let $Z \subset X$ be an irreducible component of $\Nlc(X, \Delta)$. 
Then (up to replacing $f$ by some power) we have
$$
\fibre{f}{Z} = Z.
$$

If $(X, \Delta)$ has at most log-canonical singularities, let $Z$ be an lc centre. Then (up to replacing $f$ by some power) we have
$$
\fibre{f}{Z} = Z.
$$
In this case we have $Z \not\subset R_\Delta$ where $R_\Delta$ is the logarithmic ramification divisor.
\end{lemma}

\begin{proof} 
By \eqref{logarithmic} and Remark \ref{remarklogarithmic} we have
$$
K_X + \Delta = f^* (K_X+\Delta) + R_\Delta,
$$
with $R_\Delta$ an effective Weil divisor that is $\Q$-Cartier.

Let us recall a computation from \cite[Prop.5.20]{KM98}:
let $W \subset X$ be any subvariety, and
let $\holom{\mu}{X'}{X}$ be a proper birational morphism from a normal variety $X'$ such that 
$$
K_{X'}+\mu_*^{-1}(\Delta) = \mu^*(K_X+\Delta) + R + a(E, X, \Delta) E
$$
with $R$ a $\mu$-exceptional divisor and $E$ a $\mu$-exceptional prime divisor such that $\mu(E)=W$.
Let $X''$ be the normalisation of the fibre product $X \times_X X'$ and consider the following commutative
diagram
\[
\xymatrix{
X''
\ar[r]^{f'} \ar[d]_{\mu'} & X' \ar[d]^{\mu}
\\
X \ar[r]^{f} & X
}
\]
Let $W_1, \ldots, W_r$ be the irreducible components of $\fibre{f}{W}$.
By Remark \ref{remarkdominates} every $W_i$ dominates $W$ via $f$. 
Thus for every $i \in \{1, \ldots, r\}$ the
fibre product 
$$
W_i \times_W E \subset X \times_X X'
$$ 
contains an irreducible divisorial component that surjects onto $W_i$.
Let $E_i' \subset X''$ be a prime divisor that maps onto this divisor, then we have 
$f'(E_i')=E$. Denote by $r_i$
the ramification index of $f'$ along $E_i'$. By \cite[p.160, last line]{KM98} we have
$$
a(E_i', X, \Delta-R_\Delta)+1 = r \left( a(E, X, \Delta)+1 \right).
$$ 
Since $R_\Delta$ is effective and $\Q$-Cartier, we have
$$
a(E_i', X, \Delta) \leq a(E_i', X, \Delta-R_\Delta)  
$$
with equality holding if and only if $W_i \not\subset R_{\Delta}$.
Thus we see that if $a(E, X, \Delta)<1$ (resp.  $a(E, X, \Delta)\leq1$) then we have $a(E_i', X, \Delta)<1$ (resp. $a(E_i', X, \Delta)\leq1$). Moreover we have the following implication:
\begin{equation} \label{ramify}
\mbox{If} \ a(E, X, \Delta)=1 \ \mbox{and} \ a(E_i', X, \Delta)=1, \mbox{then} \
W_i \not\subset R_{\Delta}.
\end{equation}

{\em Proof of the first statement.} 
We will argue by descending induction on the dimension of the irreducible
components of $\Nlc(X, \Delta)$. The start of the induction is trivial since there is no irreducible
component of $\Nlc(X, \Delta)$ of dimension $\dim X$.
Suppose now that every irreducible component
of $\Nlc(X, \Delta)$ of dimension at least $m+1$ is totally invariant, and let 
$Z_1, \ldots, Z_k$ be the irreducible components of $\Nlc(X, \Delta)$ of dimension $m$.

Fix a $j \in \{ 1, \ldots, k\}$, and let $Z_{j}'$ be an irreducible component of $\fibre{f}{Z_j}$.
By what precedes we have $Z_j' \subset \Nlc(X, \Delta)$. We claim that $Z_j'$ is actually an irreducible
component of $\Nlc(X, \Delta)$: if this was not the case there would be an irreducible component $W$ of $\Nlc(X, \Delta)$
such that $Z_j' \subset W$ and $\dim W \geq m+1$. Yet by our induction hypothesis $W$ is totally invariant,
so $Z_j' \subset W$ implies that $Z_j \subset W$. Thus $Z_j$ is not an irreducible component of $\Nlc(X, \Delta)$,
a contradiction.

Hence every irreducible component of $\fibre{f}{Z_j}$ is an irreducible component of dimension $m$ of $\Nlc(X, \Delta)$.
Since there are only finitely many such components, namely $Z_1, \ldots, Z_k$, we see that 
$f^{-1}$ induces a bijection on
the irreducible components of dimension $m$ of $\Nlc(X, \Delta)$. Thus some power of $f$ induces the identity.

{\em Proof of the second statement.} Since $X$ is log-canonical there exist only finitely many lc centres.
We can now repeat the proof of the first statement to see that $f^{-1}$ acts by permutation on the lc centres,
so some power induces the identity. 
An lc centre $Z$ that is totally invariant and contained in $R_\Delta$ contradicts the statement \eqref{ramify}, 
so it does not exist.
\end{proof}

\begin{lemma} \label{lemmalifting}
Let $X_1$ and $X_2$ be normal varieties, and let $\holom{f}{X_1}{X_2}$ be a finite morphism.
Let $\Delta_1$ and $\Delta_2$ be reduced effective Weil divisors on $X_1$ and $X_2$ such that
$\Delta_1 = \supp f^* \Delta_2$ and we have
$$
K_{X_1}+\Delta_1 = f^* (K_{X_2}+\Delta_2).
$$
Suppose that the pair $(X_2, \Delta_2)$ has a log-canonical model 
$\holom{\mu_2}{(Y_2, \Delta_{Y,2})}{(X_2, \Delta_2)}$.

Then the pair $(X_1, \Delta_1)$ has a log-canonical model $\holom{\mu_1}{(Y_1, \Delta_{Y, 1})}{(X_1, \Delta_1)}$, moreover 
$f$ lifts to a finite morphism
$\holom{g}{Y_1}{Y_2}$ such that
$$
K_{Y_1}+\Delta_{Y, 1} = g^* (K_{Y_2}+\Delta_{Y, 2})
$$
and $\mu_2 \circ g = f \circ \mu_1$.
\end{lemma}

Our proof follows Nakayama's argument in the surface case \cite[Lemma 2.7.6]{Nak08}.

\begin{proof}
Let $Y_1$ the normalization of the fiber product $X_1 \times_{X_2} Y_2$. Then we have a commutative diagram
$$
\xymatrix{
 Y_1 \ar[d]^{p_1} \ar[r]^{p_2} & Y_2 \ar[d]^{\mu_2} \\
 X_1    \ar[r]^{f} & X_2
}
$$
where the morphisms $p_i$ are induced by the projections from the fibre product. 
Recall that by Definition \ref{definitionlcmodel} one has
$$
\Delta_{Y, 2} = (\mu_2)_*^{-1}(\Delta_2)+E^{lc}_{\mu_2},
$$
where $E^{lc}_{\mu_2}$ is the sum of all the $\mu_2$-exceptional prime divisors taken with coefficient one.
Since $f$ and $p_2$ are finite we see that 
$$
\supp(p_2^* E^{lc}_{\mu_2})
$$
is the sum of all the $p_1$-exceptional prime divisors taken with coefficient one.
We set
$$
\Delta_{Y, 1} := (p_1^{-1})_* \Delta_1 + \supp(p_2^* E^{lc}_{\mu_2})
$$
and claim that the ramification formula
$$
K_{Y_1}+\Delta_{Y, 1} = p_2^* (K_{Y_2}+\Delta_{Y, 2})
$$
holds. Assuming this for the time being, let us see how to conclude: by \cite[Prop.5.20]{KM98} the pair
$(Y_1, \Delta_{Y, 1})$ is log-canonical. Since the morphism $p_1$ is obtained by base-changing $\mu_2$ and normalising, 
the pull-back of the $\mu_2$-ample divisor $K_{Y_2}+\Delta_{Y,2}$ is $p_1$-ample. By uniqueness of the log-canonical model
(cf. Remark \ref{remarklcmodel}) we see that $(Y_1, \Delta_{Y, 1})$ is the log-canonical model of $(X_1, \Delta_1)$.
The finite morphism $g:=p_2$ gives the lifting of $f$. 

{\em Proof of the claim.} We have $\supp f^* \Delta_2 = \Delta_1$, hence by our definition of $\Delta_{Y, 1}$
$$
\supp p_2^* \Delta_{Y, 2} = \Delta_{Y, 1}.
$$
Thus by the logarithmic ramification formula \eqref{logarithmic} we have
$$
K_{Y_1}+\Delta_{Y, 1} = p_2^* (K_{Y_2}+\Delta_{Y, 2}) + R_\Delta
$$
with $R_\Delta$ an effective divisor that has no common component with $\Delta_{Y, 1}$. 
Since by hypothesis $K_{X_1}+\Delta_1 = f^* (K_{X_2}+\Delta_2)$ it is clear that $R_\Delta$ is $p_1$-exceptional.
Since $\Delta_{Y,1}$ contains every $p_1$-exceptional prime divisors with coefficient one, the divisor $R_\Delta$ is zero.
\end{proof}

\section{Proofs of the main results}

\begin{proposition} \label{propositionbranch}
Let $X$ be a normal variety, and let $\holom{f}{X}{X}$ be an endomorphism of degree $\deg(f)>1$.
Let $\Delta$ be a reduced effective  totally invariant Weil divisor such that
$K_X+\Delta$ is $\Q$-Cartier.

Let $Z$ be an irreducible component of $\Nlc(X, \Delta)$ that is totally invariant. 
Then $Z \not\subset R_\Delta$ where $R_\Delta$ is the logarithmic branch divisor.
\end{proposition}

\begin{remark}
If $\Delta=0$ and $X$ is a surface this follows from a theorem of Wahl \cite{Wah90}, cf. also Favre \cite{Fav10}. 
More generally if $\Delta=0$ and $X$ has at most isolated singularities, we can apply \cite[Thm.B]{BDF12}
or \cite[Cor.]{Ful11}.
Our strategy is inspired by Nakayama's proof of the surface case \cite[Lemma 2.7.9]{Nak08}.
\end{remark}

\begin{proof}
Let \holom{\mu}{(Y, \Delta_Y)}{(X, \Delta)} be the log-canonical model of $(X, \Delta)$. 
By Remark \ref{remarklcmodel} we have 
\begin{equation} \label{lcmodel2}
K_{Y} + \Delta_{Y} = \mu^* (K_{X}+ \Delta) + \Delta_{Y}^{>1},
\end{equation}
where $\Delta_{Y}^{>1}$ is an antieffective divisor such that $\supp \Delta_{Y}^{>1} = \Exc(\mu)$.
Since $Z$ is an irreducible component of $\Nlc(X, \Delta)$ 
there exists at least one prime divisor $E_1$ in $Y$ that surjects onto $Z$.
Denote by $E_1, \ldots, E_k$ the irreducible components of $\supp(\Delta_Y^{>1})$ that surject onto $Z$.
Then we can write
\begin{equation} \label{minusone}
\Delta_Y^{>1} = \sum_{i=1}^k a_i E_i + E',
\end{equation}
where the $a_i$ are the log-discrepancies with respect to $(X, \Delta)$.
Since $Z$ is an irreducible component of $\Nlc(X, \Delta)$ the antieffective divisor $E'$ has the property $Z \not\subset \mu(\supp(E'))$.

We will argue by contradiction and suppose that $Z \subset R_\Delta$.

{\em Step 1. An estimate of the discrepancies.}
Let 
$$
K_X+\Delta = f^* (K_X+\Delta) + R_\Delta
$$ 
be the logarithmic ramification formula.
By Remark \ref{remarklogarithmic} the divisor $R_\Delta$ is $\Q$-Cartier and we denote by $m$
its Cartier index.
Thus the pull-back $\mu^* R_\Delta$ is well-defined and since $Z \subset R_\Delta$ we have
$$
\mult_{E_i} (\mu^* R_\Delta) \geq \frac{1}{m}
$$
for every $i=1, \ldots, k$. 
Note moreover that for all $l \in \N$ the logarithmic ramification divisor $R_{\Delta, l}$ of the $l$-th iterate $f^l$ satisfies
$$
R_{\Delta, l} = \sum_{j=0}^{l-1} (f^j)^* (R_\Delta).
$$
Since $Z \subset R_\Delta$ and $\fibre{f}{Z}=Z$ we see that $Z \subset (f^j)^* (R_\Delta)$, hence 
$$
\mult_{E_i} (\mu^* (f^j)^* R_\Delta) \geq \frac{1}{m}
$$
for all $i$ and $j$. Thus for $l$ sufficiently high we have
$\mult_{E_i} (\mu^* R_{\Delta, l}) + a_i \geq 0$.
Since our statement does not depend on the iterate of $f$ we can suppose without loss
of generality that these inequalities holds for $l=1$. Thus we have
\begin{equation} \label{estimatediscrepancies}
\mult_{E_i} (\mu^* R_{\Delta}) + a_i \geq 0 
\end{equation}
for all $i \in \{ 1, \ldots, k\}$.

{\em Step 2. Comparing the discrepancies.}
The endomorphism $f$ induces a rational map $Y \dashrightarrow Y$, we choose a resolution of the indeterminacies of
$\holom{\nu}{V}{Y}$ such that $V$ is smooth. Then we obtain a generically finite, projective, surjective morphism $\holom{g}{V}{Y}$ 
such that we have a commutative diagram
$$
\xymatrix{
V   \ar[rd]_{g} \ar[r]^{\nu} & Y  \ar[r]^{\mu} & X  \ar[d]^{f}
\\
& Y \ar[r]^{\mu}  & X
}
$$
Using the notation of Lemma \ref{lemmagenericallyfinite} we have
\begin{equation} \label{ramify1}
K_V+\Delta_V = g^* (K_{Y} + \Delta_Y) + R_g.
\end{equation} 
Note that by the definition of $\Delta_V$ we have $\Delta_Y = g(\Delta_V)$.

The pair $(Y, \Delta_Y)$ is log-canonical, so we can write
$$
K_V = \nu^* (K_{Y}+\Delta_Y) + N'
$$
where $N'$ is a divisor such that all coefficients are at least $-1$.
Thus if we set $N := N'+\Delta_V$, then
\begin{equation} \label{ramify2}
K_V + \Delta_V = \nu^* (K_Y+\Delta_Y) + N
\end{equation}
and for every irreducible component $D \subset \Delta_V$ we have
\begin{equation} \label{zerobis}
\mult_D N \geq 0.
\end{equation}
By \eqref{ramify1} and \eqref{ramify2} we have 
$$
\nu^* (K_Y+\Delta_Y) + N = g^* (K_{Y} + \Delta_Y) + R_g.
$$
Plugging in \eqref{lcmodel2} on both sides we get
$$
\nu^* (\mu^* (K_X+\Delta) + \Delta_Y^{>1}) + N = g^* (\mu^* (K_X+\Delta) + \Delta_Y^{>1}) + R_g
$$
By the logarithmic ramification formula $K_X+\Delta=f^* (K_X+\Delta)+R_\Delta$ we can simplify to
\begin{equation} \label{one}
\nu^* (\mu^* R_\Delta + \Delta_Y^{>1}) + N = g^* \Delta_Y^{>1} + R_g
\end{equation} 
Since $g(\Delta_V)=\Delta_Y$ and $\supp \Delta_Y^{>1} \subset \Delta_Y$ (cf. Remark \ref{remarklcmodel}) there exists 
a prime divisor $D \subset \Delta_V$ such that $g(D)=E_1$. Let us first observe that
\begin{equation} \label{three}
\mu(\nu(D))=Z.
\end{equation}
Indeed by our commutative diagram
$$
f(\mu(\nu(D)))  = \mu(g(D))=\mu(E_1)=Z,
$$  
hence $\mu(\nu(D))$ is contained in $\fibre{f}{Z}$ which by hypothesis is $Z$.
Since $Z$ is irreducible and $\mu(\nu(D))$ has dimension at least $\dim Z$ (it surjects via $f$ on $Z$), we get the equality \eqref{three}.

By  Lemma \ref{lemmagenericallyfinite} we know that $\Delta_V$ and $R_g$ do not have common components,
so $\mult_D R_g=0$. Since $\Delta_Y^{>1}$ is antieffective and its support contains $E_1$, we obtain
\begin{equation} \label{two}
\mult_D (g^* \Delta_Y^{>1} + R_g)<0.
\end{equation}
Consider now the decomposition
$\Delta_Y^{>1} = \sum_{i=1}^k a_i E_i + E'$ introduced in \eqref{minusone}.
We have $Z \not\subset \mu(\supp(E'))$ and $\mu(\nu(D))=Z$ by \eqref{three}, 
so we see that $\nu(D) \not\subset \supp(E')$.
Since $\mu^* R_\Delta+\Delta_{Y}^{>1}$ is $\Q$-Cartier this implies that $\mu^* R_\Delta + \sum a_i E_i$ is $\Q$-Cartier in the generic point
of $\nu(D)$. By the inequalities \eqref{estimatediscrepancies} we know that 
$$
\mu^* R_\Delta + \sum a_i E_i  
$$
is an effective divisor,
so we obtain
$$
\mult_D \nu^* (\mu^* R_\Delta + \Delta_{Y}^{>1}) = \mult_{\nu(D)_{\gen}} (\mu^* R_\Delta + \sum a_i E_i)  \geq 0.
$$
Yet by \eqref{zerobis} this implies that
$$
\mult_D  (\nu^* (\mu^* R_\Delta + \Delta_{Y}^{>1}) + N) \geq 0,
$$
so by \eqref{one} we have a contradiction to \eqref{two}.
\end{proof}

\begin{proof}[Proof of Theorem \ref{theoremmainlocalpair}]
By Lemma \ref{lemmainvariance} we can suppose (up to replacing $f$ by some iterate)
that all the irreducible components of $\Nlc(X, \Delta)$ are totally invariant. 
Let $Z$ be such an irreducible component, then by Proposition \ref{propositionbranch}
we have $Z \not\subset R_\Delta$, where $R_\Delta$ is the logarithmic branch divisor.
We will now argue by contradiction and suppose that there exists an irreducible component $Z \subset \Nlc(X, \Delta)$
such that the induced endomorphism \holom{f|_Z}{Z}{Z} satisfies 
\begin{equation} \label{smaller}
\deg(f|_Z)<\deg(f).
\end{equation}

Let $(\tilde X, Z_{\gen})$ be the germ of the normal variety $X$ in the generic point $Z_{\gen} \subset X$,
and denote by
$$
\holom{\tilde f}{(\tilde X, Z_{\gen})}{(\tilde X, Z_{\gen})}
$$
the induced endomorphism. Set $\tilde \Delta:=\Delta|_{\tilde X}$, then the finite morphism $\tilde f$ 
\'etale in codimension one, i.e. we have
\begin{equation} \label{logetale1}
K_{\tilde X}+\tilde \Delta = (\tilde f)^* (K_{\tilde X}+\tilde \Delta).
\end{equation}
Let $\holom{\tilde \mu}{(\tilde Y, \Delta_{\tilde Y})}{(\tilde X, \tilde \Delta)}$
be the log-canonical model. By Lemma \ref{lemmalifting} the finite morphism $\tilde f$
lifts to a finite morphism $\holom{g}{\tilde Y}{\tilde Y}$ such that
\begin{equation} \label{logetale2}
K_{\tilde Y}+\Delta_{\tilde Y} = g^* (K_{\tilde Y}+\Delta_{\tilde Y})
\end{equation}
and $\mu \circ g = \tilde f \circ \mu$.

Since $Z_{\gen}$ is an irreducible component of $\Nlc(\tilde X, \tilde \Delta)$ and
the $\mu$-exceptional locus has pure codimension one (cf. Remark \ref{remarklcmodel}),
there exists at least one prime divisor $E_1$ in $\tilde Y$ that surjects onto $Z_{\gen}$.
Let $E_1, \ldots, E_k$ be the prime divisors in \fibre{\mu}{Z_{\gen}} that surject onto $Z_{\gen}$,
then $g^{-1}$ acts by permutation on the set of divisors $\{ E_1, \ldots, E_k \}$.
Thus (up to replacing $\tilde f$ and hence $g$ by some iterate) we can assume that $g^{-1}$ acts as the identity.
Let now 
$$
\holom{g|_{E_1}}{E_1}{E_1}
$$
be the induced endomorphism. We claim that we have 
$$
\deg(g|_{E_1}) = \deg (f|_Z).
$$
Assuming this for the time being, let us see how to conclude:
since $\deg(f)=\deg(g)$ our claim and \eqref{smaller} implies that $\deg(g|_{E_1})<\deg(g)$. Thus $E_1$
is contained in the branch divisor of $g$ and we have
\begin{equation} \label{Doneramified}
g^* E_1 = r E_1
\end{equation}
with $r>1$. By Remark \ref{remarklcmodel} we have 
$$
K_{\tilde Y} + \Delta_{\tilde Y} = \mu^* (K_{\tilde X}+\tilde \Delta) + \Delta_{\tilde Y}^{>1},
$$
where $\Delta_{\tilde Y}^{>1}$ is an antieffective divisor such that $\supp \Delta_{\tilde Y}^{>1} = \Exc(\mu)$. Plugging this into
\eqref{logetale2} we obtain
$$
\mu^* (K_{\tilde X}+\tilde \Delta) + \Delta_{\tilde Y}^{>1} = g^* \mu^* (K_{\tilde X}+\tilde \Delta) + g^* \Delta_{\tilde Y}^{>1}.
$$
Yet by \eqref{logetale1} this simplifies to
$$
\Delta_{\tilde Y}^{>1} = g^* \Delta_{\tilde Y}^{>1}.
$$
Since $\supp \Delta_{\tilde Y}^{>1} = \Exc(\mu)$ it contains the divisor $E_1$. Thus by restricting the equation above
to $E_1$ we obtain $g^* E_1=E_1$, a contradiction to \eqref{Doneramified}.

{\em Proof of the claim.} We have a commutative diagram
$$
\xymatrix{
 E_1 \ar[d]_{\mu|_{E_1}} \ar[r]^{g|_{E_1}} & E_1 \ar[d]^{\mu|_{E_1}} \\
 Z _{\gen}   \ar[r]^{f|_{Z_{\gen}}} & Z_{\gen}
}
$$
Let $F_1$ be a general fibre of $\mu|_{E_1}$ and set
$F_2:=g|_{E_1}(F_1)$. Then $F_2$ is a general  $\mu|_{E_1}$-fibre, in particular
$F_1$ and $F_2$ are homologous. 
Set \holom{\tilde g}{F_1}{F_2}.
By \eqref{logetale2} we have
$$
(K_{\tilde Y}+\Delta_{\tilde Y})^{\dim F_1} \cdot F_1 = (g^* (K_{\tilde Y}+\Delta_{\tilde Y}))^{\dim F_1} \cdot F_1
= \deg(\tilde g) (K_{\tilde Y}+\Delta_{\tilde Y})^{\dim F_1} \cdot F_2.
$$
Since $F_1$ and $F_2$ are homologous we have
$$
(K_{\tilde Y}+\Delta_{\tilde Y})^{\dim F_1} \cdot F_2
=
(K_{\tilde Y}+\Delta_{\tilde Y})^{\dim F_1} \cdot F_1.
$$
Moreover $K_{\tilde Y}+\Delta_{\tilde Y}$ is ample on $F_1$, so these intersection numbers are not zero. Thus we obtain that
$$
\deg \tilde g = 1.
$$
By the commutative diagram above this implies the claim.
\end{proof}

\begin{corollary} \label{corollarypolarisedpair}
Let $X$ be a normal projective variety, and let $\holom{f}{X}{X}$ be a polarised endomorphism of degree $\deg(f)>1$.
Let $\Delta$ be a reduced effective  totally invariant Weil divisor such that $K_X+\Delta$ is $\Q$-Cartier.

Then the pair $(X, \Delta)$ is log-canonical. Moreover if $Z$ is an lc centre of $(X, \Delta)$, then (up to replacing $f$ by some iterate) $Z$ is totally invariant. In this case we have 
$Z \not\subset R_\Delta$ where $R_\Delta$ is the logarithmic ramification divisor.
\end{corollary}

Note that the case $\Delta=0$ of this statement corresponds to Corollary \ref{corollarypolarised}.

\begin{proof}
The endomorphism $f$ is polarised, so there exists an ample divisor $H$ such that $f^* H \simeq m H$
with $m>1$. Thus if $Z \subset X$ is a totally invariant subvariety, the endomorphism $\holom{f|_Z}{Z}{Z}$ is polarised by $H|_Z$.
In particular we have
$$
\deg(f|_Z) = m^{\dim Z} < m^{\dim X} = \deg(f).
$$
By Theorem \ref{theoremmainlocalpair} this implies that $\Nlc(X, \Delta)$ is empty.
The second part of the statement follows from Lemma \ref{lemmainvariance}.
\end{proof}

For inductive purposes the following non-normal version should be useful.

\begin{corollary} \label{corollarypolarisedpairnonnormal}
Let $X$ be a projective variety that is $S_2$ and whose codimension one points are either regular points or ordinary nodes\footnote{$X$ is demi-normal in the sense of Koll\'ar.}. Let $\holom{f}{X}{X}$ be a polarised endomorphism of degree $\deg(f)>1$. Let $\Delta$ be a reduced effective totally invariant Weil divisor such that $K_X+\Delta$ is $\Q$-Cartier and no irreducible component of $\Delta$ is contained in the non-normal locus. 

Then the pair $(X, \Delta)$ is semi-log-canonical. 
\end{corollary}

\begin{proof} Let \holom{\nu}{\tilde X}{X} be the normalisation.
Let $D \subset X$ be the divisor defined by the conductor of the normalisation,
and let $\tilde \Delta$ be the divisorial part of $\fibre{\nu}{\Delta}$. Then we have
$$
K_{\tilde X} + \tilde \Delta + D = \nu^* (K_X+\Delta),
$$
so $K_{\tilde X} + \tilde \Delta + D$ is $\Q$-Cartier. Note that $D$ is reduced since $X$ has ordinary nodes 
in codimension one.

By the universal property of the normalisation, the endomorphism $f$ lifts to an endomorphism
$\holom{\tilde f}{\tilde X}{\tilde X}$. Moreover the divisor $D$ is totally invariant (cf. Prop.5.4. in 
the arXiv version of \cite{NZ10}).
By Corollary \ref{corollarypolarisedpair} the pair $(\tilde X, \tilde \Delta+D)$ is log-canonical.
Thus $(X, \Delta)$ is semi-log-canonical. 
\end{proof}

\def\cprime{$'$}

\end{document}